\documentclass[12pt]{article}

\usepackage{hyperref}

\usepackage{amsfonts, amsmath, amssymb, amsthm}
\usepackage{a4}

\usepackage{blkarray}
\usepackage{bbm}
\usepackage{upgreek}

\usepackage{bm}

\usepackage{graphicx}

 \usepackage{xcolor}

 \definecolor{Refkey}{RGB}{255,127,0}
 \definecolor{Labelkey}{RGB}{127,0,255}
 \makeatletter 
  \def\SK@refcolor{\color{Refkey}}
  \def\SK@labelcolor{\color{Labelkey}}
 \makeatother

  \definecolor{mdg}{RGB}{0,177,0} 
  \definecolor{mdb}{RGB}{0,0,191}
  \definecolor{mddb}{RGB}{0,0,91}
  \definecolor{mdy}{RGB}{255,69,0} 
  \definecolor{gray}{RGB}{99,99,99}

\DeclareMathOperator{\Aut}{Aut}

\newtheorem{proposition}{Proposition}
\newtheorem{theorem}[proposition]{Theorem}
\newtheorem{lemma}{Lemma}
\newtheorem{corollary}[proposition]{Corollary}

\theoremstyle{definition}
\newtheorem{definition}{Definition}
\newtheorem{example}{Example}

\theoremstyle{remark}

\allowdisplaybreaks[3]

\title{Self-similarity in cubic blocks of $\mathcal R$-operators}
\author{Igor G. Korepanov}

\date{September 2022 --- September 2023}

\begin{document}

 \sloppy

\maketitle

\medskip

\begin{abstract}
Cubic blocks are studied assembled from linear operators~$\mathcal R$ acting in the tensor product of $d$ linear ``spin'' spaces. Such operator is associated with a linear transformation~$A$ in a vector space over a field~$F$ of a finite characteristic~$p$, like ``permutation-type'' operators studied by Hietarinta. One small difference is that we do not require $A$ and, consequently, $\mathcal R$ to be invertible; more importantly, no relations on~$\mathcal R$ are required of the type of Yang--Baxter or its higher analogues.

It is shown that, in $d=3$ dimensions, a $p^n\times p^n\times p^n$ block decomposes into the tensor product of operators similar to the initial~$\mathcal R$. One generalization of this involves commutative algebras over~$F$ and allows to obtain, in particular, results about spin configurations determined by a \emph{four}-dimen\-sional~$\mathcal R$. Another generalization deals with introducing Boltzmann weights for spin configurations; it turns out that there exists a non-trivial self-similarity involving Boltzmann weights as well.
\end{abstract}


\section{Introduction}\label{s:i}

\subsection[$\mathcal R$-operators]{$\bm{\mathcal R}$-operators}\label{ss:n}

In this paper, a \emph{$d$-dimen\-sional $\mathcal R$-operator} is, by definition, a linear operator acting in the tensor product $\bigotimes _{i=1}^d \mathcal V_i$ of $d$ finite-dimen\-sional linear spaces~$\mathcal V_i$ over the field~$\mathbb C$ of complex numbers, or field~$\mathbb R$ of real numbers.

Study of such operators is typically associated with exactly solvable models in mathematical physics, or with constructing knot invariants in topology, in which cases the $\mathcal R$-operator is supposed to satisfy an ``$n$-simplex equation'', such as Yang--Baxter~\cite{Jimbo,Wu} (where $n=2$) or Zamolodchikov tetrahedron~\cite{Zamolodchikov} (where $n=3$).

In two-dimen\-sional statistical physics, one of the typical cases is where such $\mathcal R$-operators are placed in the vertices of a square lattice, while two orthogonal lines intersecting at each~$\mathcal R$ represent the two spaces $\mathcal V_1$ and~$\mathcal V_2$ in whose tensor product it acts nontrivially. Yang--Baxter equation may then make it possible to calculate some thermodynamical quantities of the model. Similar r\^ole is played in three dimensions by the cubic lattice and Zamolodchikov tetrahedron equation.

Below we attempt, however, to study algebraic structures related to such lattices \emph{directly}, using $\mathcal R$-operators that are \emph{not} supposed to obey any of the mentioned equations. It is natural to begin with a simple kind of $\mathcal R$-operators; first, we recall \emph{permutation-type} operators studied by Hietarinta~\cite{Hietarinta}. These operators mapped any basis vector of the mentioned space $\bigotimes _{i=1}^d \mathcal V_i$ again into its basis vector, in a bijective way, performing thus a \emph{permutation} of the basis. Likewise, we introduce \emph{basis-induced} operators in Definition~\ref{d:R} in almost the same way, but not requiring the mapping of the basis into itself to be bijective.

Matrix entries of basis-induced operators~$\mathcal R$---they may be called ``local Boltzmann weights'' in the context of statistical physics---are clearly either 1 or~0. We will see, however, that they allow us already to calculate some statistical quantities, such as the number of permitted spin configurations on a finite lattice determined by~$\mathcal R$ and some kinds of boundary conditions.

In Section~\ref{s:w}, we introduce \emph{non-trivial}---real positive---Boltzmann weights, in a certain ``non-local'' manner, retaining the term ``basis-induced'' for operators with matrix entries 0 and~1.

\subsection{Motivation}\label{ss:m}

Our main interest lies in studying statistical properties of cubic lattices in higher dimensions $d>2$. In contrast with the case $d=2$, where Yang--Baxter equation has been applied very successfully, only a limited success has been achieved while applying its higher analogues for $d\ge 3$. Apparently, these higher analogues are too restrictive.

These restrictions look especially salient for basis-induced operators. For instance, for $d=3$, such operators are classified in~\cite[Section~5]{Hietarinta}. The reader can see that \emph{all} corresponding matrices~$A$ are \emph{block triangular}, and this means that their three-dimen\-sion\-ality cannot be called full-fledged: there is either one dimension that does not depend on the two others, or two dimensions that do not depend on the remaining one. This is why we are trying to find a different approach.

We will be dealing with \emph{cubic blocks}---cubic fragments of a lattice. One obvious motivation for studying algebraic structure of blocks of $\mathcal R$-operators comes from Kadanoff--Wilson theory~\cite{Kadanoff,Wilson} of critical phenomena in statistical physics, namely from the ``block spin'' idea of Kadanoff~\cite{Kadanoff}. In that theory, a few cubic fragments are united into a greater cube, and then a transformation of spin variables is done, aiming at separating ``less important'' variables from ``more important''. Then, this step is repeated, with an even greater cube as the result, and so on.

It turns out that something like that can be done in our algebraic construction as well: cubic blocks---to which $\mathcal R$-operators correspond---are assembled into a greater cubic block, then, after a transformation of spin variables, the result splits into the tensor product of $\mathcal R$-operators of the same or almost the same kind as the initial one! One difference with the Kadanoff--Wilson theory is that all resulting $\mathcal R$-operators look equally important, at least at the current stage of the development of the theory.

\subsection{What is done in this paper}\label{ss:w}

As we have already said, we consider here one of the simplest kinds of $\mathcal R$-operators---``basis-induced''. In contrast with Hietarinta's work~\cite{Hietarinta}, where he gave a classification of such operators satisfying some $n$-simplex equations, we do not, however, require them to satisfy any specific equations.

Below,
\begin{itemize}\itemsep 0pt
 \item in Section~\ref{s:g}, generalities are recalled about basis-induced $\mathcal R$-operators obtained from linear transformations,
 \item in Section~\ref{s:2d}, we consider a two-dimen\-sional toy example of our theory,
 \item in Section~\ref{s:3d}, we consider a decomposition of a (three-dimen\-sional) $p\times p\times p$ block of $\mathcal R$-operators related to a $3\times 3$ matrix~$A$ with entries in a field of characteristic~$p$,
 \item in Section~\ref{s:3dd}, we specify the mentioned decomposition for the case $p=2$. We
 write out explicit formulas for the spin transformation and do an important observation that they work also in a more general case where matrix~$A$ entries belong not just to a finite field, but to a commutative algebra; we study two specific cases of iterative block making, or ``evolution'', as we call it,
 \item in Section~\ref{s:4d}, we show how some specific commutative algebras allow us to calculate even four-dimen\-sional statistical quantities using a proper version of ``three-dimen\-sional'' methods of Section~\ref{s:3dd},
 \item in Section~\ref{s:w}, we demonstrate how to introduce non-trivial Boltzmann weights into the theory, in such way that self-similarity holds for them as well. This is done on a few specific examples, and involves a study of what can be called `most symmetric probability distributions',
 \item finally, in Section~\ref{s:d}, we discuss our results and possible directions of future research.
\end{itemize}

\section[Basis-induced $\mathcal R$-operators, cubic blocks, and permitted spin configurations]{Basis-induced $\bm{\mathcal R}$-operators, cubic blocks, and permitted spin configurations}\label{s:g}

In Sections \ref{s:g}--\ref{s:4d}, there is no essential difference whether we consider $\mathbb C$-linear or $\mathbb R$-linear $\mathcal R$-operators. For brevity, we speak of them as $\mathbb C$-linear, keeping in mind that $\mathbb C$ can be changed to~$\mathbb R$, if needed, with no problem.

\subsection[Basis-induced $\mathcal R$-operators from linear transformations in a direct sum]{Basis-induced $\bm{\mathcal R}$-operators from linear transformations in a direct sum}\label{ss:p}

\begin{definition}\label{d:Q}
\emph{Quantum space}~$\mathcal X$ corresponding to a finite set~$X$ is the linear space over~$\mathbb C$ whose basis consists of all elements $x\in X$.
\end{definition}

In other words, $\mathcal X$ consists of formal $\mathbb C$-linear combinations of elements $x\in X$.

We will be interested in the case where $X=V$ is a finite-dimen\-sional linear space over a finite field~$F$. Then, if $F$ has $q$ elements and $V$ is $m$-dimen\-sional, the quantum space~$\mathcal V$ corresponding to~$V$ is $q^m$-dimen\-sional.

Also, it is quite clear that the quantum space corresponding to a Cartesian product $X\times Y$ is the \emph{tensor} product $\mathcal X \otimes \mathcal Y$ of the corresponding quantum sets. Recall that in the case of vector spaces, Cartesian product is the same as direct sum.

\begin{definition}\label{d:R}
Let $X$ be a finite set, and $f\colon\;X\to X$ be its mapping into itself. Then \emph{basis-induced $\mathcal R$-operator} corresponding to~$f$ is defined as the $\mathbb C$-linear endomorphism of the corresponding quantum space~$\mathcal X$ sending each of its basis vectors~$x$ into its basis vector $y=f(x)$.
\end{definition}

The case interesting for us will be where $X=V=\bigoplus _{i=1}^d V_i$ is a direct sum of finite-dimen\-sional linear spaces~$V_i$ over a (fixed) finite field~$F$, and $f=A$ is a linear operator acting in~$V$. Operator~$\mathcal R$ acts then in the tensor product $\mathcal V=\bigotimes _{i=1}^d \mathcal V_i$ of the corresponding quantum spaces.

\subsection{Convention: operators and matrices act from the right; matrices act hence on rows}\label{ss:conv}

If, again, $V=\bigoplus _{i=1}^d V_i$ and $A$ is a linear operator acting in~$V$, and if a basis $\{ \mathsf e_1^{(i)},\ldots,\mathsf e_{\dim V_i}^{(i)} \}$ is given in each~$V_i$---and hence their union forms a basis in~$V$---then we can identify $A$ with its matrix (which will also be denoted~$A$ when this does not lead to a misunderstanding), and say that $\mathcal R$ corresponds to matrix~$A$ with entries in field~$F$.

In this connection, we adopt the following convention.
In this paper, linear transformations like $A$ or~$\mathcal R$ will be written as acting on relevant vectors from the \emph{right}! When identifying an operator~$A$ with its matrix, we identify vectors in~$V$ with \emph{row} vectors, on which matrix~$A$ acts again from the right. A given basis in~$V$ is identified, accordingly, with the standard row basis:
\begin{equation*}
\begin{pmatrix} 1 & 0 & \dots & 0 \end{pmatrix},\quad \begin{pmatrix} 0 & 1 & \dots & 0 \end{pmatrix}, \quad \ldots, \quad \begin{pmatrix} 0 & 0 & \dots & 1 \end{pmatrix}. 
\end{equation*}

\subsection{Products of operators acting in different spaces}\label{ss:pd}

We list here some properties of $\mathcal R$-operators that are quite obvious from their definition:
\begin{itemize}\itemsep 0pt
 \item $\mathcal R$-operator corresponding to the identity operator $A=\mathbbm 1_V$ in $F$-linear space~$F$ acts also as identity $\mathcal R=\mathbbm 1_{\mathcal V}$ in its quantum space~$\mathcal V$,
 \item $\mathcal R$-operator corresponding to the product $A_1 \cdots A_m$ of several linear operators acting in the same $F$-linear space~$V$ is the product $\mathcal R = \mathcal R_1 \cdots \mathcal R_m$ of the corresponding $\mathcal R$-operators,
 \item $\mathcal R$-operator corresponding to the direct sum $\bigoplus _{j=1}^p A_j$ of several linear operators acting each in its own $F$-linear space~$V_j$ is the tensor product $\bigotimes _{j=1}^p \mathcal R_j$ of the corresponding $\mathcal R$-operators and acts, accordingly, in the tensor product $\mathcal V = \bigotimes _{j=1}^p \mathcal V_j$ of the corresponding quantum spaces.
\end{itemize}

Let now there be some number~$N$ of linear spaces~$V_i$, and some number~$m$ of $F$-linear operators~$A_k$ such that each~$A_k$ acts in the direct sum of only \emph{some} of spaces~$V_i$. We want to give sense to the product of all these~$A_k$.

The standard well-known way of doing this is as follows: extend the action of each~$A_k$ onto the whole direct sum $\bigoplus _{i=1}^N V_i$ of \emph{all} spaces~$V_i$ as follows:
\begin{equation}\label{eA}
A_k \to A_k \oplus \mathbbm 1_{\mathrm{remaining}},
\end{equation}
where $\mathbbm 1_{\mathrm{remaining}}$ means the identity operator acting in the direct sum of those spaces where $A_k$ does \emph{not} act. Then, the usual product of the right-hand sides of~\eqref{eA} is taken for $k=1,\ldots,m$. Following tradition, we can write this product simply as $\prod _{k=1}^m A_k$, tacitly identifying each~$A_k$ with the r.h.s.\ of~\eqref{eA}.

The same applies to the product of the corresponding $\mathcal R$-operators, with the understanding that direct sums are replaced with tensor products: each $\mathcal R_k$ is tacitly identified with $\mathcal R_k \otimes \mathbbm 1_{\mathrm{remaining}}$, where $\mathbbm 1_{\mathrm{remaining}}$ means this time the identity operator acting in the tensor product of those spaces where $\mathcal R_k$ does not act.

\subsection{Cubic blocks: definition}\label{ss:b}

Consider the integer lattice~$\mathbb Z^d$ within the $d$-dimen\-sional real space with coordinates~$x_i$:
\begin{equation*}
\mathbb Z^d \subset \mathbb R^d \ni \mathbf x = (x_1,\ldots,x_d) ,
\end{equation*}
and introduce the following partial order on it:
\begin{equation}\label{preceq}
\mathbf x \preceq \mathbf y \quad\Leftrightarrow\quad x_i \le y_i \quad \text{for all }\,i.
\end{equation}

Let there be given finite-dimen\-sional $F$-linear spaces~$V_i$, \ $i=1,\ldots,d$---remember that $F$ is a finite field---and a linear operator~$A$ acting in~$\bigoplus _{i=1}^d V_i$.

Consider then the part~$C$ of the lattice contained in the (closed) cube with an integer edge length~$(l-1)$, as follows:
\begin{equation*}
C\subset \mathbb Z^d\colon \quad 0\le x_i\le l-1, \quad i=1,\ldots,d.
\end{equation*}
Note that there are $l$ lattice points along each \emph{edge} of the cube.

Consider all straight lines parallel to coordinate axes and going through the points of~$C$. To each such line, if it is parallel to the $i$-th axis, we put in correspondence a copy of space~$V_i$, and to each point $\mathbf x\in C$ we put in correspondence a copy of the operator~$A$, acting in those copies of spaces~$V_i$ that correspond to the lines intersecting at~$\mathbf x$. We denote such copy~$A^{(\mathbf x)}$.

\begin{definition}\label{d:b}
Cubic \emph{$\underbrace {l\times \dots \times l}_d$ block} of operators~$A$ is the product~$R$ of all copies~$A^{(\mathbf x)}$, taking in any order that agrees with the partial order relation:
\begin{equation}\label{bl}
R = \prod _{\mathbf x\in C} A^{(\mathbf x)}, \qquad \mathbf x \prec \mathbf y \; \Rightarrow \; A^{(\mathbf x)} \text{ precedes } A^{(\mathbf y)}.
\end{equation}
\end{definition}

\begin{proposition}\label{p:pp}
$R$ taken according to~\eqref{bl} exists, and does not depend on a specific choice of order of the copies~$A^{(\mathbf x)}$.
\end{proposition}

\begin{proof}
Existence: $R$ can be taken as
\begin{equation}\label{po}
R = \prod _{k=0}^{d(l-1)}\, \prod _{x_1+\dots+x_d=k} A^{(\mathbf x)}.
\end{equation}
Every two operators $A^{(\mathbf x)}$ and~$A^{(\mathbf y)}$ in the \emph{inner} product act nontrivially (not as identity) in two disjoint $d$-tuples of copies of spaces~$V_i$, and hence commute; so, the inner product can be taken in any order.

Uniqueness: let there be a product~\eqref{bl} where there are two \emph{neighbors}\/ $A^{(\mathbf x)}$ and~$A^{(\mathbf y)}$ such that $x_1+\dots+x_d > y_1+\dots+y_d$ but $\mathbf x$ and~$\mathbf y$ are not comparable (note that $\mathbf x \prec \mathbf y$ cannot be). Then, a small reasoning shows that $A^{(\mathbf x)}$ and~$A^{(\mathbf y)}$ act again in disjoint sets of spaces and hence commute, so we can swap their places without changing~$R$. Repeating such steps, we arrive at the form~\eqref{po}.
\end{proof}

In this context, we often call $A$ or its copies ``bricks'' from which block~$R$ is assembled.

Clearly, this construction can be generalized to parallelepipeds, and to many other cases if needed.

\subsection{Cubic blocks as bricks}\label{ss:bb}

Recall that a copy of $F$-linear space~$V_i$ was put in correspondence, in the previous Subsection~\ref{ss:b}, to each straight line parallel to the $i$-th coordinate axis and going through (some) points of our ``integer cube''~$C$. We can now introduce ``thick space''~$\mathbf V_i$ corresponding to $i$-th dimension as the direct sum of all such copies:
\begin{equation}\label{Vb}
\mathbf V_i = \bigoplus_{\substack{ \text{all lines where}\\ \text{copies of }V_i\text{ belong}}} (\text{copies of }V_i).
\end{equation}

Block~$R$~\eqref{bl} is then an $F$-linear operator acting in~$\bigoplus _{i=1}^{d} \mathbf V_i$. That is, $R$ has, essentially, the same nature as our ``brick''~$A$; in particular, $R$ determines a (``thick'') basis-induced $\mathcal R$-operator, and $R$ can be used as a brick for constructing an even thicker block in the same way as $R$ was constructed from~$A$.

Then, an ``evolution'' can be launched by iterating this block making construction. Remarkably, the resulting large blocks reveal, in many cases, an unexpectedly simple algebraic structure.

\subsection{Spins and permitted spin configurations}\label{ss:pc}

Straight lines introduced in Subsection~\ref{ss:b} are divided into segments, or ``edges'', by the points of the ``integer cube''~$C$, some of these edges being half-infinite. Every edge has actually a direction---that of the corresponding coordinate axis, so $d$ incoming and $d$ outgoing edges meet at each vertex (\,=\,point) $\mathbf c\in C$.

Suppose now that a vector is assigned to each edge, namely, an element of linear space~$V_i$ if the edge is parallel to the $i$-th axis.

\begin{definition}\label{d:pv}
Vectors assigned to edges are \emph{consistent around vertex~$\mathbf c$} if the vectors at its incoming edges are transformed by~$A$ into the vectors at its outgoing edges. If this holds for all $\mathbf c\in C$, these vectors are said to form a \emph{permitted configuration}.
\end{definition}

We say that the vectors attached to incoming, resp.\ outgoing, edges at a vertex~$\mathbf c$ form the \emph{input}, resp.\ \emph{output}, of the corresponding~$A$. Similarly, all the vectors attached to all the incoming, resp.\ outgoing, \emph{half-infinite} (that is, not inner) edges of a block~$R$, form the \emph{input}, resp.\ \emph{output}, of~$R$. If only edges parallel to the $i$-th axis are taken, we speak of the \emph{$i$-th input/output}.

Any conditions/restrictions imposed on the input and/or output vectors are called \emph{boundary conditions}.

In this paper, spaces~$V_i$ will typically be one-dimen\-sional, thus, every (copy of) $V_i$ can be identified with the field~$F$. In this case, we often call elements of~$F$ attached to edges ``spins''.

We give the following natural definition.

\begin{definition}\label{d:ps}
\emph{Permitted spin configuration with given boundary conditions} is a permitted spin configuration in the sense of Definition~\ref{d:pv}, satisfying also the specified boundary conditions.
\end{definition}

An interesting statistical quantity is the number of permitted spin configuration with given boundary conditions. In this paper, we consider combinations of the following linear conditions, each involving the spins at two opposite faces of our integer cube---$i$-th input and $i$-th output:
\begin{enumerate}\itemsep 0pt
 \item \label{i:p} periodic boundary conditions along the $i$-th axis: the spins at the $i$-th output face must coincide with the corresponding spins at the $i$-th input face,
 \item \label{i:g} ``$i$-th input all zeros'': all spins at the $i$-th input are fixed at zeros, while the spins at the $i$-th output are free,
 \item \label{i:f} ``free $i$-th input and output'': no conditions on the corresponding spins.
\end{enumerate}

\begin{proposition}\label{p:nc}
The number of permitted spin configurations with toric (periodic in all $d$ dimensions) boundary conditions for a block~$R$ is~$|F|^{\dim E_1}$, where $|F|$ is the number of elements in~$F$, and $E_1$ is the eigenspace of~$R$ corresponding to eigenvalue~1.
\end{proposition}

\begin{proof}
Indeed, the opposite spins must coincide, which means that the output vector must coincide with the input vector. This happens exactly when each of them belongs to~$E_1$, and the number of vectors in~$E_1$ is~$|F|^{\dim E_1}$.
\end{proof}

\subsection{Gauge transformations}\label{ss:gauge}

Suppose we have chosen bases in spaces~$V_i$ (or $\mathbf V_i$, as in~\eqref{Vb})and thus realized them as row spaces. If we pass to different bases, row vectors $v_i\in V_i$ undergo transformations $v_i \mapsto v_i g_i$ with some matrices~$g_i\in \Aut V_i$. Then, the following conjugation with a block-diagonal matrix applies to matrix~$R$ (or~$A$):
\begin{equation}\label{gt}
R \mapsto G^{-1} R G,\qquad G = \begin{pmatrix} g_1 & & \\ & \ddots & \\ & & g_d \end{pmatrix}.
\end{equation}

Below, we call this \emph{gauge transformation}, and call ``$R$ and~$G^{-1} R G$ are (gauge) equivalent''. We can also say that ``matrix~$R$ becomes $G^{-1} R G$ in a proper basis''.

We will see that matrices corresponding to our blocks may simplify drastically under some gauge transformations. In this connection, we make the following simple but important remark: \emph{boundary conditions chosen according to any combination of the above items \ref{i:p}--\ref{i:f} remain the same under a gauge transformation}.

\section{Two dimensions: a toy theory}\label{s:2d}

\subsection[Decomposition of a $2\times 2$ block in characteristic two]{Decomposition of a $\bm{2\times 2}$ block in characteristic two}\label{ss:2*2}

The first observation, and a toy example of what we will do in the next sections, comes with making a $2\times 2$ block of matrices
\begin{equation}\label{A2}
A = \begin{pmatrix} a & b \\ c & d \end{pmatrix}
\end{equation}
Let the entries in~\eqref{A2} belong to a field~$F$, and let $F$ be, for a moment, of characteristic~0. Make a $2\times 2$ block as in Figure~\ref{f:2d}.
\begin{figure}
 \centering
 \includegraphics[scale=0.8]{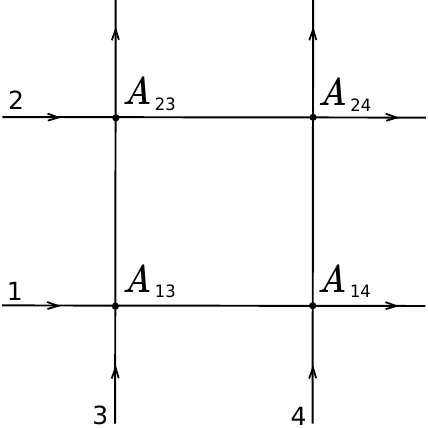}
 \caption{$2\times 2$ block~$R$ in two dimensions. Numbers denote copies of field~$F$; matrix~$R$ acts, accordingly, in~$F^4$}
 \label{f:2d}
\end{figure}
The product of four copies of~$A$ in Figure~\ref{f:2d} can be written as
\begin{equation}\label{RA}
R = A_{13} A_{14} A_{23} A_{24},
\end{equation}
where numbers denote the spaces where $A$ acts nontrivially, and we find that
\begin{equation}\label{R2}
R = \begin{pmatrix} R_{\mathbf{11}} & R_{\mathbf{12}} \\ R_{\mathbf{21}} & R_{\mathbf{22}} \end{pmatrix} =
\left(\begin{array}{cc|cc} a^2 & 2 a b c & b d & a b d + b ^2 c\\
0 & a^2 & b & a b\\ \hline
a c & a c d+b c^2 & d^2 & 2 b c d\\
c & c d & 0 & d^2 \end{array} \right),
\end{equation}
where the thick subscript~$\mathbf 1$ unites usual subscripts 1 and~2, while~$\mathbf 2$ unites 3 and~4.

Explicit expression~\eqref{R2} leads immediately to some observations: if we change~$F$ to a field of \emph{characteristic two}, then
\begin{itemize}\itemsep 0pt
 \item $R_{\mathbf{11}}$ and~$R_{\mathbf{22}}$ become scalar (multiples of the identity matrix~$\mathbbm 1$),
 \item $R_{\mathbf{12}}R_{\mathbf{21}}$ becomes equal to $R_{\mathbf{21}}R_{\mathbf{12}}$
 \item and, moreover, $R_{\mathbf{12}}R_{\mathbf{21}}=R_{\mathbf{21}}R_{\mathbf{12}}= b^2 c^2 \mathbbm 1$ is also scalar!
\end{itemize}

\begin{proposition}\label{p:2d}
If entries in matrix~$A$ belong to a field of characteristic~2, then, in a proper basis in each thick space, $R$ is the direct sum of two copies of $\begin{pmatrix} a^2 & b^2 \\ c^2 & d^2 \end{pmatrix}$. Namely, if we take basis vectors $\mathsf e_1=\begin{pmatrix} 1 & 0 \end{pmatrix}$ and $\mathsf e_2=\begin{pmatrix} 0 & 1 \end{pmatrix}$ in the first (horizontal) thick space, and vectors $\mathsf e_1 R_{\mathbf{12}}/b^2$ and $\mathsf e_2 R_{\mathbf{12}}/b^2$ in the second (vertical) thick space, then $R$ transforms into
\begin{equation*}
\begin{pmatrix} a^2 & 0 & b^2 & 0 \\ 0 & a^2 & 0 & b^2 \\ c^2 & 0 & d^2 & 0 \\ 0 & c^2 & 0 & d^2 \end{pmatrix}.
\end{equation*}
\end{proposition}

\begin{proof}
This follows directly from the above observations.
\end{proof}

So, essentially, matrix~$A$, when making a $2\times 2$ block, undergoes the \emph{Frobenius automorphism}, and multiplies into two copies. The corollary below follows immediately.

\begin{corollary}\label{c:n}
The iterated block (see Subsection~\ref{ss:b}) is, after\/ $n$ iterations, equivalent to $2^n$ copies of matrix $\begin{pmatrix} a^{2^n} & b^{2^n} \\ c^{2^n} & d^{2^n} \end{pmatrix}$.
\qed
\end{corollary}

\subsection{Other characteristics}\label{ss:oc}

The observations in the previous subsection can be generalized, with some moderate effort, to any finite field characteristic~$p$. The ``brick'' matrix~$A$ has again form~\eqref{A2}, while the analogue of~$R$ of Figure~\ref{f:2d} will now be a $p\times p$ block. We introduce, as before, thick subscripts $\mathbf 1$ and~$\mathbf 2$, corresponding to the horizontal and vertical directions, respectively, and write
\begin{equation}\label{R2p}
R = \begin{pmatrix} R_{\mathbf{11}} & R_{\mathbf{12}} \\ R_{\mathbf{21}} & R_{\mathbf{22}} \end{pmatrix},
\end{equation}
as in~\eqref{R2}, but each $R_{\mathbf{ij}}$ is now a $p\times p$ matrix.

\begin{proposition}\label{p:Rp}
In~\eqref{R2p},
\begin{equation}\label{Rij}
R_{\mathbf{11}} = a^p \mathbbm 1_p, \qquad R_{\mathbf{22}} = d^p \mathbbm 1_p, \qquad R_{\mathbf{12}}R_{\mathbf{21}} = R_{\mathbf{21}}R_{\mathbf{12}} = b^p c^p \mathbbm 1_p .
\end{equation}
\end{proposition}

\begin{proof}
Each matrix entry of any~$R_{\mathbf{ij}}$ is a sum of products of matrix entries $a,b,c,d$ of~$A$. We would like to write these letters $a,b,c,d$ in each such product in the same order as we used in Definition~\ref{d:b}. In such order, as one can quite easily see,
\begin{itemize}\itemsep 0pt
 \item $a$ can be followed only by $a$ or~$b$,
 \item $b$ can be followed only by $c$ or~$d$,
 \item $c$ can be followed only by $a$ or~$b$,
 \item $d$ can be followed only by $c$ or~$d$.
\end{itemize}
This can be explained graphically by saying that $a$ corresponds to going through one~$A$ along the horizontal direction, $d$ corresponds similarly to going along the vertical direction, while $b$ changes horizontal to vertical, and $c$---vice versa, see Figure~\ref{f:abcd}.
\begin{figure}
 \centering
 \includegraphics[scale=1.3]{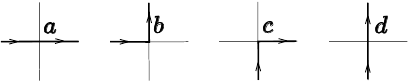}
 \caption{Matrix entries of~$A$ and corresponding fragments of ways through the block~$R$}
 \label{f:abcd}
\end{figure}

For $R_{\mathbf{11}}$, moreover, such product must begin with either $a$ or~$b$, and end with either $a$ or~$c$. Hence, it can be thought of as consisting of some strings of one of the two following kinds: $a$ or~$b d^l c$, where $l \ge 0$. The whole number of such strings must be~$p$, as we must cross our $p\times p$ block~$R$ in the horizontal direction. Moreover, there must be at least one~$a$---otherwise, we would get outside our block in the vertical direction.

One possible product for $p=5$ is depicted in Figure~\ref{f:5*5}.
\begin{figure}
 \centering
 \includegraphics[scale=1.3]{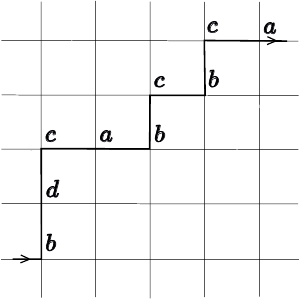}
 \caption{Graphic representation of one product of matrix entries of~$A$ entering as a summand in one matrix entry of~$R_{\mathbf{11}}$, for $p=5$}
 \label{f:5*5}
\end{figure}

For a fixed matrix entry of~$R_{\mathbf{11}}$, we unite all such products having the same strings of the form~$b d^l c$, going in the same order, in an equivalence class. There are then $\binom{p}{k}$ products in such equivalence class, where $k$ is the number of the~$a$'s, $1\le k \le p$.

If $k=p$, there are no strings~$b d^l c$, and we get a diagonal entry~$a^p$ of~$R_{\mathbf{11}}$.

For a non-diagonal entry, $1\le k < p$, and $\binom{p}{k}$ is a multiple of~$p$, that is, zero modulo~$p$. We have proved the first equality in~\eqref{Rij}; also, the second equality is proved similarly.

Consider now the product $R_{\mathbf{12}} R_{\mathbf{21}}$. We argue that its only nonzero matrix entries are diagonal, and correspond to ``ladders'' such as depicted in Figure~\ref{f:BC}.
\begin{figure}
 \centering
 \includegraphics[scale=1.3]{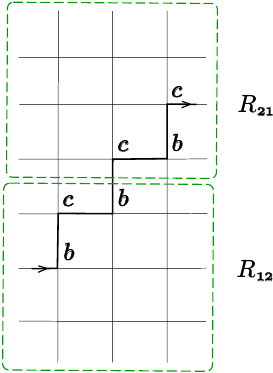}
 \caption{One of the ``ladders'' giving nonzero entries of $R_{\mathbf{12}} R_{\mathbf{21}}$ for $p=3$}
 \label{f:BC}
\end{figure}

Indeed, if there is at least one~$a$ in the product, then we can include it in an equivalence class giving together zero due to the same reasons as we considered above.

If there are no~$a$'s, but at least one~$d$ is present, denote $k$ the number of such~$d$'s. These $d$'s are situated between the neighboring $b$'s and~$c$'s:
\begin{equation}\label{bdddc}
b d \dots d c ,
\end{equation}
and there are $p$ strings~\eqref{bdddc}, because we must do $p$ steps in the horizontal direction. Also, $k<p$, because otherwise we would go too far in the vertical direction.

Allocating our $d$'s in the $p$ strings of type~\eqref{bdddc} in different ways, we see that there are $\binom{k+p-1}{k}$ ways to do so, which again vanishes modulo~$p$, because $1\le k < p$.

The proof ends with noting that the ladders in Figure~\ref{f:BC} give of course diagonal entries~$b^p c^p$.
\end{proof}

\begin{corollary}\label{c:pn}
In a proper basis in each thick space, $R$ becomes the direct sum of $p$ copies of matrix $\begin{pmatrix} a^p & b^p \\ c^p & d^p \end{pmatrix}$.
\end{corollary}

\begin{proof}
A straightforward modification of the proof of Proposition~\ref{p:2d}.
\end{proof}

\section{Blocks in three dimensions}\label{s:3d}

\subsection{The key results}\label{ss:3dr}

We start with a ``brick''
\begin{equation}\label{A}
A = \begin{pmatrix} a_{11} & a_{12} & a_{13} \\ a_{21} & a_{22} & a_{23} \\ a_{31} & a_{32} & a_{33} \end{pmatrix}
\end{equation}
having entries in a field~$F$ of a finite characteristic~$p$.

Then we construct a $p\times p\times p$ block. For $p=2$, it is shown in Figure~\ref{f:3d},
\begin{figure}
 \centering
 \includegraphics[scale=1]{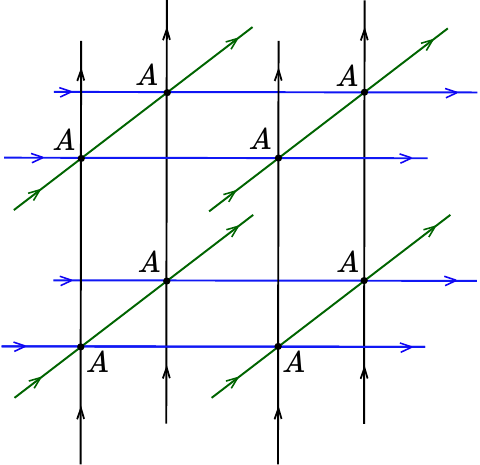}
 \caption{$2\times 2\times 2$ block~$R$ in three dimensions}
 \label{f:3d}
\end{figure}
where four lines of each of the three colors (and directions) correspond together to one of the three ``thick'' spaces. We denote this block as
\begin{equation}\label{R}
R = \begin{pmatrix} R_{11} & R_{12} & R_{13} \\ R_{21} & R_{22} & R_{23} \\ R_{31} & R_{32} & R_{33} \end{pmatrix} ,
\end{equation}
where we do not write bold subscripts any longer, but each $R_{ij}$ acts of course from the $i$-th \emph{thick} space into the $j$-th thick space, like in Section~\ref{s:2d}.

The analogue of Proposition~\ref{p:Rp} still holds.

\begin{proposition}\label{p:B3d}
$\mathrm{(a)}$ Diagonal entries $R_{11}$, $R_{22}$ and~$R_{33}$ in~\eqref{R} are scalar, namely, $R_{ii}=a_{ii}^p \cdot \mathbbm 1$.

$\mathrm{(b)}$ Besides, $R_{kl}R_{lk}=R_{lk}R_{kl}$ for $1\le k<l\le 3$, and all these products are also scalars, namely, $a_{kl}^2 a_{lk}^2 \cdot \mathbbm 1$.
\end{proposition}

\begin{proof}
A straightforward modification of the proof of Proposition~\ref{p:Rp}.
\end{proof}

There is one more remarkable property, concerning the product
\begin{equation}\label{RRR}
R_{12}R_{23}R_{31}.
\end{equation}
We formulate it as the following Theorem~\ref{th:p}.

\begin{theorem}\label{th:p}
For matrix~$A$~\eqref{A} belonging to a field~$F$ of any finite characteristic~$p$, the Jordan form of the product~\eqref{RRR} consists of cells with eigenvalue $a_{12}^p a_{23}^p a_{31}^p$, their total dimension (``algebraic multiplicity'' of the mentioned eigenvalue) being\/~$\dfrac{p(p+1)}{2}$, and cells with eigenvalue $a_{13}^p a_{32}^p a_{21}^p$, their total dimension being\/~$\dfrac{p(p-1)}{2}$.

Moreover, if the following condition holds:
\begin{equation}\label{ineq}
a_{12}a_{23}a_{31} \ne a_{13}a_{32}a_{21} ,
\end{equation}
then $R$ is\/ \emph{diagonalizable}---all Jordan cells have dimension one.
\end{theorem}

The proof of Theorem~\ref{th:p} will be given in the form of Lemmas~\ref{l:k}--\ref{l:m} in the next Subsection~\ref{ss:3dl}. Here, we prove its following corollary---shedding light on the key role of the product~\eqref{RRR}.

\begin{corollary}\label{p:diag}
Assuming~\eqref{ineq}, consider the following two matrices, each being the transpose of the other:
\begin{equation}\label{a^p}
\mathsf r^{\mathrm T} = 
\begin{pmatrix} a_{11}^p & a_{21}^p & a_{31}^p \\ a_{12}^p & a_{22}^p & a_{32}^p \\ a_{13}^p & a_{23}^p & a_{33}^p \end{pmatrix}
 \qquad \text{and} \qquad
\mathsf r = 
\begin{pmatrix} a_{11}^p & a_{12}^p & a_{13}^p \\ a_{21}^p & a_{22}^p & a_{23}^p \\ a_{31}^p & a_{32}^p & a_{33}^p \end{pmatrix}.
\end{equation}
Choosing a proper basis in each thick space, we can bring\/ $R$ to the form of the direct sum of\/ $\dfrac{p(p-1)}{2}$ matrices~$\mathsf r^{\mathrm T}$ and\/ $\dfrac{p(p+1)}{2}$ matrices~$\mathsf r$.
\end{corollary}

\begin{proof}
First, take such basis in the first thick space where $R_{12}R_{23}R_{31}$ is diagonal. Denote this basis $\{ \mathsf f_1^{(1)}, \ldots, \mathsf f_{p(p-1)/2}^{(1)}, \mathsf e_1^{(1)}, \ldots, \mathsf e_{p(p+1)/2}^{(1)} \}$, where basis vectors~$\mathsf f_i^{(1)}$ correspond to eigenvalue $a_{13}^p a_{32}^p a_{21}^p$, and basis vectors~$\mathsf e_i^{(1)}$---to eigenvalue $a_{12}^p a_{23}^p a_{31}^p$; the superscript~${}^{(1)}$ means the first space. Then, for each~$\mathsf f_i^{(1)}$, define a basis vector in each of the spaces 2 and~3 as follows:
\begin{equation*}
\mathsf f_i^{(2)}=\frac{ \mathsf f_i^{(1)}R_{12} }{a_{21}^p}, \qquad \mathsf f_i^{(3)}=\frac{ \mathsf f_i^{(1)}R_{13} }{a_{31}^p} ,
\end{equation*}
while for each~$\mathsf e_i^{(1)}$, define a basis vector in each of the spaces 2 and~3 as follows:
\begin{equation*}
\mathsf e_i^{(2)}=\frac{ \mathsf e_i^{(1)}R_{12} }{a_{12}^p}, \qquad \mathsf e_i^{(3)}=\frac{ \mathsf e_i^{(1)}R_{13} }{a_{13}^p} .
\end{equation*}
A matrix~$\mathsf r^{\mathrm T}$ corresponds then to each triple $\{ \mathsf f_i^{(1)}, \mathsf f_i^{(2)}, \mathsf f_i^{(3)} \}$, while a matrix~$\mathsf r$---to each triple $\{ \mathsf e_i^{(1)}, \mathsf e_i^{(2)}, \mathsf e_i^{(3)} \}$.
\end{proof}

\subsection{Proof of Theorem~\ref{th:p}}\label{ss:3dl}

\begin{lemma}\label{l:k}
 \begin{equation}\label{kvadrat}
 R_{12}R_{23}R_{31} + R_{13}R_{32}R_{21} = (a_{12}^p a_{23}^p a_{31}^p + a_{13}^p a_{32}^p a_{21}^p )\, \mathbbm 1_p .
 \end{equation}
\end{lemma}

\begin{proof}
First, we note that the l.h.s.\ of~\eqref{kvadrat} corresponds to the ``square diagram'' of Figure~\ref{f:k}
\begin{figure}
 \centering
 \includegraphics[scale=2.5]{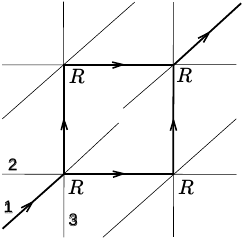}
 \caption{Square diagram. 1, 2, 3 are numbers of thick spaces}
 \label{f:k}
\end{figure}
with four matrices~$R$. The two summands in the l.h.s.\ of~\eqref{kvadrat} correspond to the two paths from the input (lower left oblique edge) to the output (upper right oblique edge).

Our~$R$'s act, however, in thick spaces, each consisting of $p$ ``thin'' spaces (where matrices~$A$ act). In terms of thin spaces, a matrix entry of the l.h.s.\ of~\eqref{kvadrat} corresponds to a path along the edges of a cubic lattice with $p\times 2p \times 2p$ vertices.

Consider the intersection of such path with a plane perpendicular to the first direction (hence, parallel to the plane of Figure~\ref{f:k}) and containing some of the vertices (there are $4p^2$ vertices in such plane of course). Such intersection may look, for instance, as in one of the pictures of Figure~\ref{f:ts}.
\begin{figure}
 \centering
 \includegraphics[scale=2]{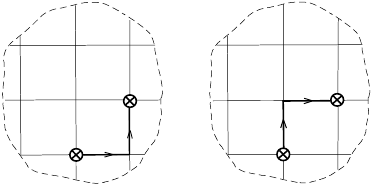}
 \caption{Two (of the many) possible ways within one plane perpendicular to direction~1. Small circles with oblique crosses stay where there are edges perpendicular to the plane of the picture and directed away from us}
 \label{f:ts}
\end{figure}

Any path of the considered kind intersects all $p$ planes perpendicular to direction~1. We now declare that two such paths belong to the same \emph{equivalence class} if their corresponding \emph{multisets}, consisting each of the $p$ intersections of the path with the $p$ planes, coincide. That is, there is a permutation of the planes that converts one sequence of $p$ intersections into the other; the intersections are considered to within parallel translations.

It is not hard to see that the matrix entries corresponding to all paths belonging to a given equivalence class are the same. One more easily established fact is that if there are at least two \emph{different} intersections in a multiset, then the cardinality of the corresponding equivalence class is a multiple of~$p$. Hence, the contribution of the whole equivalence class vanishes, given the characteristic~$p$ of the field where all matrix entries belong.

What remains is the paths whose intersections with all the planes are \emph{the same}. Additionally, their inputs and outputs must fit within the ``thick'' input and output of Figure~\ref{f:k}. With these requirements, the only two possibilities are: either all intersections are as in the l.h.s.\ of Figure~\ref{f:ts}, or as in its r.h.s.

The paths corresponding to the first possibility give together $a_{12}^p a_{23}^p a_{31}^p \mathbbm 1_p$, while those corresponding to the second possibility give  $a_{13}^p a_{32}^p a_{21}^p \mathbbm 1_p$.
\end{proof}

\begin{lemma}\label{l:e}
The product $R_{12}R_{23}R_{31}$ admits Jordan decomposition over the same field\/~$F$ where the entries~$a_{ij}$ of matrix~$A$ belong. Moreover, all its eigenvalues are either $a_{12}^p a_{23}^p a_{31}^p$ or $a_{13}^p a_{32}^p a_{21}^p$.
\end{lemma}

\begin{proof}
As the sum $R_{12}R_{23}R_{31} + R_{13}R_{32}R_{21}$ is, according to~\eqref{kvadrat}, a scalar matrix, both its summands can be brought into a Jordan form together, at least if we take a proper extension of the field\/$F$, and allow the \emph{nondiagonal} entries in the Jordan form to be~$-1$ rather than~$1$ for one of the summands.

Moreover, the \emph{product} of our two summands, due to item~$\mathrm{(b)}$ of Proposition~\ref{p:B3d}, is again a scalar matrix, namely,
\begin{equation}\label{RRRRRR}
R_{12}R_{23}R_{31} \cdot R_{13}R_{32}R_{21} = a_{12}^p a_{23}^p a_{31}^p \cdot a_{13}^p a_{32}^p a_{21}^p \cdot \mathbbm 1_p .
\end{equation}

If $\lambda$ is an eigenvalue of~$R_{12}R_{23}R_{31}$, and $\mu$---the eigenvalue of~$R_{13}R_{32}R_{21}$ for the same eigenvector, then it follows immediately from \eqref{kvadrat} and~\eqref{RRRRRR} that
\begin{equation}\label{l,m}
 \begin{aligned}
\text{either} & \quad\; \lambda = a_{12}^p a_{23}^p a_{31}^p \quad \text{and} \quad \mu = a_{13}^p a_{32}^p a_{21}^p, \\
\text{or} & \quad\; \lambda = a_{13}^p a_{32}^p a_{21}^p \quad \text{and} \quad \mu = a_{12}^p a_{23}^p a_{31}^p .
 \end{aligned}
\end{equation}
It follows also that no extension of the field\/~$F$ is actually needed.
\end{proof}

\begin{lemma}\label{l:d}
If condition~\eqref{ineq} holds, then matrices $R_{12}R_{23}R_{31}$ and $R_{13}R_{32}R_{21}$ are \emph{diagonalizable}---all Jordan cells are of dimension one.
\end{lemma}

\begin{proof}
Consider a supposed Jordan cell for $R_{12}R_{23}R_{31}$ of dimension more than one. If it has unities as its non-diagonal entries (as it is usually required in a definition of a Jordan cell), then $R_{13}R_{32}R_{21}$ has minus unities at the corresponding places, according to~\eqref{kvadrat}. Then, the \emph{product} of these cells has non-diagonal entries $-\lambda + \mu$. These entries must, however, vanish, because of~\eqref{RRRRRR}. Hence, $\lambda$ must be the same as~$\mu$, but \eqref{l,m} shows that it cannot be under the condition~\eqref{ineq}.
\end{proof}

\begin{lemma}\label{l:m}
For matrix $R_{12}R_{23}R_{31}$, algebraic multiplicities of the eigenvalues\/ $a_{12}^p a_{23}^p a_{31}^p$ and\/ $a_{13}^p a_{32}^p a_{21}^p$ are\/ $\dfrac{p(p+1)}{2}$ and\/ $\dfrac{p(p-1)}{2}$, respectively.
\end{lemma}

\begin{proof}
First, we can consider the entries $a_{11}, a_{12}, \ldots, a_{33}$ of matrix~$A$ as \emph{indeterminates} over the prime field~$\mathbb F_p$, taking our field~$F$ as $F = \mathbb F_p (a_{11}, a_{12}, \ldots, a_{33})$. According to Lemma~\ref{l:e}, the eigenvalues of $R_{12}R_{23}R_{31}$ are $a_{12}^p a_{23}^p a_{31}^p$ and $a_{13}^p a_{32}^p a_{21}^p$, with some algebraic multiplicities.

These multiplicities cannot change, however, if we specialize $a_{11}, a_{12}, \ldots, a_{33}$ to be \emph{any specific} elements of~$\mathbb F_p$ or its any extension. This means that, in order to find these multiplicities, we can take any simple particular case of~$A$, for instance, set $a_{12} = a_{23} = a_{31} = 1$, and $a_{ij} = 0$ for the remaining six entries~$a_{ij}$.

For such~$A$, there is no problem in finding eigenvectors of $R_{12}R_{23}R_{31}$: take one input and set its spin to one, while all the others to zero. In terms of our pictures, eigenvalue~$1$ will correspond to the spins on the SW--NE diagonal and below, while eigenvalue~$0$---to the spins above that diagonal, see Figure~\ref{f:muh}.
\begin{figure}
 \centering
 \includegraphics[scale=2]{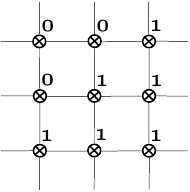}
 \caption{Eigenvalues of $R_{12}R_{23}R_{31}$ for $A=\begin{pmatrix}0 & 1 & 0\\ 0 & 0 & 1\\ 1 & 0 & 0\end{pmatrix}$ and $p=3$}
 \label{f:muh}
\end{figure}
\end{proof}

Now Theorem~\ref{th:p} has clearly been proven.

\section{Blocks in three dimensions: detailed calculations in characteristic two}\label{s:3dd}

\subsection{Diagonalizable case}\label{ss:dc}

In characteristic two, a calculation shows that there are the following four row eigenvectors for both operators $R_{12}R_{23}R_{31}$ and $R_{13}R_{32}R_{21}$ acting in the first thick space~$\mathbf V_1$ (see Subsection~\ref{ss:bb})---we denote them $\mathsf f^{(1)}$, $\mathsf e_1^{(1)}$, $\mathsf e_2^{(1)}$ and~$\mathsf e_3^{(1)}$, in accordance with the notations in the proof of Corollary~\ref{p:diag}, and write them together as the rows of a matrix:
\begin{gather}
\mbox{\small $ \displaystyle \begin{pmatrix} \mathsf f^{(1)} \\ \mathsf e_1^{(1)} \\ \mathsf e_2^{(1)} \\ \mathsf e_3^{(1)} \end{pmatrix} = $ } \nonumber \\[0.5ex]
 \label{thick_1}
\mbox{\footnotesize $ \displaystyle
\begin{pmatrix} a_{21}a_{31} & a_{31}(a_{21}a_{33}+a_{23}a_{31}) & 
  a_{21}(a_{21}a_{32}+a_{22}a_{31}) & 
  (a_{21}a_{33}+a_{23}a_{31})(a_{21}a_{32}+a_{22}a_{31}) \\
0 & 0 & a_{12} & a_{12}a_{33}+a_{13}a_{32} \\
0 & a_{13} & 0 & a_{12}a_{23}+a_{13}a_{22} \\
1 & a_{33} & a_{22} & a_{22}a_{33}+a_{23}a_{32} \end{pmatrix}
 $ }
\end{gather}
Here, $\mathsf f^{(1)}$ corresponds to eigenvalue $a_{13}^2 a_{32}^2 a_{21}^2$ of $R_{12}R_{23}R_{31}$ and eigenvalue $a_{12}^2 a_{23}^2 a_{31}^2$ of $R_{13}R_{32}R_{21}$, while each of $\mathsf e_1^{(1)}$, $\mathsf e_2^{(1)}$ and~$\mathsf e_3^{(1)}$ corresponds, vice versa, to eigenvalue $a_{12}^2 a_{23}^2 a_{31}^2$ of $R_{12}R_{23}R_{31}$ and eigenvalue $a_{13}^2 a_{32}^2 a_{21}^2$ of $R_{13}R_{32}R_{21}$.

In this subsection, we are considering the case where condition~\eqref{ineq} holds. We note that the determinant of matrix~\eqref{thick_1} is
\begin{equation}\label{db}
(a_{12}a_{23}a_{31} - a_{13}a_{32}a_{21})^2.
\end{equation}
hence \eqref{ineq} implies that the rows of~\eqref{thick_1} form a \emph{basis} in the first thick space.

Taking this basis, and choosing also bases in the two remaining thick spaces in accordance with Section~\ref{s:3d} (see again the proof of Corollary~\ref{p:diag}), namely:
\begin{gather}
\mbox{\small $ \displaystyle \begin{pmatrix} \mathsf f^{(2)} \\ \mathsf e_1^{(2)} \\ \mathsf e_2^{(2)} \\ \mathsf e_3^{(2)} \end{pmatrix} = $ } \nonumber \\[0.5ex]
 \label{thick_2}
\mbox{\footnotesize $ \displaystyle
\begin{pmatrix} a_{12} a_{32} & 
      a_{32} ( a_{12} a_{33}+a_{13} a_{32} ) & 
      a_{12} ( a_{11} a_{32}+a_{12} a_{31} ) & 
      ( a_{12} a_{33}+a_{13} a_{32} ) ( a_{11} a_{32}+a_{12} a_{31} ) \\
   1 & a_{33} & a_{11} & a_{11} a_{33}+a_{13} a_{31} \\ 
   0 & a_{23} & 0 & a_{11} a_{23}+a_{13} a_{21} \\ 
   0 & 0 & a_{21} & a_{21} a_{33}+a_{23} a_{31} \end{pmatrix}
 $ }
\end{gather}
in the second thick space, and
\begin{gather}
\mbox{\small $ \displaystyle \begin{pmatrix} \mathsf f^{(3)} \\ \mathsf e_1^{(3)} \\ \mathsf e_2^{(3)} \\ \mathsf e_3^{(3)} \end{pmatrix} = $ } \nonumber \\[0.5ex]
 \label{thick_3}
\mbox{\footnotesize $ \displaystyle
\begin{pmatrix} a_{13} a_{23} & a_{23} ( a_{12} a_{23}+a_{13} a_{22} ) & 
      a_{13} ( a_{11} a_{23}+a_{13} a_{21} ) & 
      ( a_{12} a_{23}+a_{13} a_{22} ) ( a_{11} a_{23}+a_{13} a_{21} ) \\ 
   0 & a_{32} & 0 & a_{11} a_{32}+a_{12} a_{31} \\ 
   1 & a_{22} & a_{11} & a_{11} a_{22}+a_{12} a_{21} \\ 
   0 & 0 & a_{31} & a_{21} a_{32}+ a_{22} a_{31} \end{pmatrix}
 $ }
\end{gather}
in the third thick space, we can check directly that $R$ acts as in~\eqref{RFrob}.

Note that the determinants of matrices \eqref{thick_2} and~\eqref{thick_3} are given by the same expression~\eqref{db} (which can be checked, for instance, by a direct calculation), so the rows of \eqref{thick_2} and~\eqref{thick_3} form bases indeed.

We can now specialize Corollary~\ref{p:diag} for the case of characteristic two as follows.
 
\begin{proposition}\label{p:d}
Let the entries of the ``brick'' matrix~$A$~\eqref{A} belong to a finite field~$F$ of characteristic $p=2$, and be such that \eqref{ineq} holds. Then $R$ can be represented---namely, in the basis given by vectors \eqref{thick_1}, \eqref{thick_2} and~\eqref{thick_3} in the respective thick spaces---as the following direct sum:
\begin{equation}\label{RFrob}     
\begin{pmatrix} a_{11}^2 & a_{21}^2 & a_{31}^2 \\ a_{12}^2 & a_{22}^2 & a_{32}^2 \\ a_{13}^2 & a_{23}^2 & a_{33}^2 \end{pmatrix}
\, \oplus \, \mathbbm 1_3 \otimes
\begin{pmatrix} a_{11}^2 & a_{12}^2 & a_{13}^2 \\ a_{21}^2 & a_{22}^2 & a_{23}^2 \\ a_{31}^2 & a_{32}^2 & a_{33}^2 \end{pmatrix}.
\end{equation}
\qed
\end{proposition}

Tensor product ``$\mathbbm 1_3 \otimes \text{matrix}$'' in~\eqref{RFrob} means of course the same as the direct sum of three such matrices.

\begin{corollary}\label{c:d}
The evolution---that is, iterative block making---of a brick~$A$~\eqref{A}, in characteristic~$p=2$ and with condition~\eqref{ineq}, leads, after $n$ steps, to a direct sum of $2^{2n-1}+2^{n-1}$ matrices~$\tilde A$ and $2^{2n-1}-2^{n-1}$ matrices~$\tilde A^{\mathrm T}$, where tilde over a matrix means that all its entries have been raised into the power~$2^n$.
\end{corollary}

\begin{proof}
Suppose we have, at some stage, the direct sum of $n_i$ bricks~$\tilde A$ (like the \emph{initial} matrix~$A$) and $n_t$ bricks~$\tilde A^{\mathrm T}$ (like its \emph{transpose}). If we write these numbers together as a row vector, then Proposition~\ref{p:d} can be re-written as the statement that the following transformation happens with this vector after each step:
\begin{equation*}
\begin{pmatrix} n_i & n_t \end{pmatrix} \mapsto \begin{pmatrix} n_i & n_t \end{pmatrix} Q, \qquad \text{where} \quad Q = \begin{pmatrix} 3 & 1 \\ 1 & 3 \end{pmatrix}.
\end{equation*}
One can see then that, indeed,
\begin{equation*}
\begin{pmatrix} 1 & 0 \end{pmatrix} Q^n = \begin{pmatrix} 2^{2n-1}+2^{n-1} & 2^{2n-1}-2^{n-1} \end{pmatrix}.
\end{equation*}
\end{proof}

\subsection{Bricks with elements in a commutative algebra}\label{ss:a}

As soon as we have explicit formulas \eqref{thick_1}, \eqref{thick_2} and~\eqref{thick_3} for basis vectors in the thick spaces on which the action of block~$R$~\eqref{R} is given either by~\eqref{RFrob} or its transpose, we can generalize our construction at once as follows.

Let each ``thin'' space be now~$V_i=F^n$; recall that we write its elements as $n$-rows $\begin{pmatrix} x_1 & \dots & x_n \end{pmatrix}$ of elements of field~$F\ni x_i$. Let, then, the entries~$a_{ij}$ of~$A$~\eqref{A} belong to a \emph{commutative subalgebra~$\mathcal A$ of the $n\times n$ matrix algebra} over~$F$. Denote
\begin{equation}\label{d}
d\, \stackrel{\mathrm{def}}{=} \,\det ( a_{12}a_{23}a_{31} - a_{13}a_{32}a_{21} ).
\end{equation}

\begin{proposition}\label{p:4n4n}
Each of the determinants of matrices \eqref{thick_1}, \eqref{thick_2} and~\eqref{thick_3}, regarded as $4n\times 4n$ matrices with entries in~$F$, equals\/~$d^2$.
\end{proposition}

\begin{proof}
This follows from the fact that the determinant of any of the matrices \eqref{thick_1}, \eqref{thick_2} and~\eqref{thick_3} is again given by the old expression~\eqref{db}, if we consider these matrices as $4\times 4$ matrices with entries in~$\mathcal A$.
\end{proof}

\begin{proposition}\label{p:b4n}
Assuming the following generalization of condition~\eqref{ineq}:
\begin{equation}\label{dB}
d \ne 0,
\end{equation}
bases in the three $4n$-dimen\-sional thick spaces can be chosen as follows: take the standard basis
\begin{equation}\label{rb}
\begin{pmatrix} 1 & 0 & \dots & 0 \end{pmatrix}, \quad\; \begin{pmatrix} 0 & 1 & \dots & 0 \end{pmatrix}, \quad\; \dots\, , \quad\; \begin{pmatrix} 0 & 0 & \dots & 1 \end{pmatrix},
\end{equation}
in\/~$F^n$, and multiply each of these rows from the right by the first \emph{block} row of the relevant matrix \eqref{thick_1}, \eqref{thick_2} or~\eqref{thick_3}, then similarly by the second, third and fourth row.

Then, the action of~$R$ is again given by~\eqref{RFrob}, with the understanding that there are now block matrices there.
\end{proposition}

\begin{proof}
Condition~\eqref{dB} guarantees that we indeed obtain bases in the thick spaces. Then, taking into account that rows~\eqref{rb} form together an identity matrix, we see that all the calculations of Subsection~\ref{ss:dc} are again applicable, with the only understanding that $a_{ij}$ are now elements of a commutative algebra rather than just of field~$F$.
\end{proof}

\subsection{Non-diagonalizable case}\label{ss:ndc}

We now consider the case where the inequality~\eqref{ineq} does \emph{not} hold. To be exact, we limit ourself to the case
\begin{equation}\label{symm}
a_{12}a_{23}a_{31} = a_{13}a_{32}a_{21} \ne 0.
\end{equation}

\begin{proposition}\label{p:s}
Matrix~$A$~\eqref{A} with condition~\eqref{symm} can be made \emph{symmetric} by a gauge transform (see~\eqref{gt}):
\begin{equation*}
A\mapsto G^{-1}AG, \qquad G = \begin{pmatrix} g_1 & 0 & 0 \\ 0 & g_2 & 0 \\ 0 & 0 & g_3 \end{pmatrix}
\end{equation*}
\end{proposition}

\begin{proof}
Enough to take $g_1=1$, \ $g_2=\sqrt{a_{21}/a_{12}}\,$, \ $g_3=\sqrt{a_{31}/a_{13}}\,$, keeping in mind that there is always a square root of an element of a finite field of characteristic~2.
\end{proof}

In view of Proposition~\ref{p:s}, we assume below that $A$ is \emph{already} symmetric:
\begin{equation}\label{cs}
A = \begin{pmatrix} a_{11} & a_{12} & a_{13} \\ a_{12} & a_{22} & a_{23} \\ a_{13} & a_{23} & a_{33} \end{pmatrix}.
\end{equation}

The four rows in any of \eqref{thick_1}, \eqref{thick_2} or~\eqref{thick_3} become then linearly dependent. So, we introduce the following vectors in our three thick spaces:
\begin{align*}
& \mathsf g^{(1)} = \begin{pmatrix} 0 & 0 & 0 & a_{12}a_{13}a_{23}^2 \end{pmatrix}, \\
& \mathsf g^{(2)} = \frac{\mathsf g^{(1)} R_{12}}{a_{12}^2} = \begin{pmatrix} 0 & a_{13}a_{23}^2 & 0 & a_{11}a_{13}a_{23}^2 \end{pmatrix}, \\
& \mathsf g^{(3)} = \frac{\mathsf g^{(1)} R_{13}}{a_{13}^2} = \begin{pmatrix} 0 & a_{12}a_{23}^2 & 0 & a_{11}a_{12}a_{23}^2 \end{pmatrix},
\end{align*}
and choose the following bases in these spaces:
\begin{equation*}
\begin{pmatrix} \mathsf e_1^{(1)} \\ \mathsf e_2^{(1)} \\ \mathsf g^{(1)} \\ \mathsf f^{(1)} \end{pmatrix}, \qquad
\begin{pmatrix} \mathsf e_1^{(2)} \\ \mathsf e_2^{(2)} \\ \mathsf g^{(2)} \\ \mathsf f^{(2)} \end{pmatrix} 
 \quad \text{and} \quad
\begin{pmatrix} \mathsf e_1^{(3)} \\ \mathsf e_2^{(3)} \\ \mathsf g^{(3)} \\ \mathsf f^{(3)} \end{pmatrix}.
\end{equation*}

\begin{proposition}\label{p:t}
$R$ decomposes into the direct sum of two matrices~\eqref{cs}, with the change $a_{ij}\mapsto a_{ij}^2$ of each entry, and one $6\times 6$ matrix
\begin{equation}\label{T}
\begin{pmatrix} a_{11}^2\cdot \mathbbm 1_2 & \; a_{12}^2\cdot \mathbbm 1_2 \; & a_{13}^2\cdot \mathbbm 1_2 \\
           a_{12}^2\cdot \mathbbm 1_2 & \; a_{22}^2\cdot \mathbbm 1_2 \; & a_{23}^2\cdot T \\
           a_{13}^2\cdot \mathbbm 1_2 & \; a_{23}^2\cdot T \; & a_{33}^2\cdot \mathbbm 1_2 \\ \end{pmatrix},
\end{equation}
where
\begin{equation}\label{te}
\mathbbm 1_2 = \begin{pmatrix} 1 & 0 \\ 0 & 1 \end{pmatrix}, \qquad T = \begin{pmatrix} 1 & 1 \\ 0 & 1 \end{pmatrix}.
\end{equation}
\end{proposition}

\begin{proof}
This follows from the fact that
\begin{equation}\label{eef}
\mathsf e_1^{(i)} R_{ij} = a_{ij}^2 \mathsf e_1^{(j)}, \quad\;
\mathsf e_2^{(i)} R_{ij} = a_{ij}^2 \mathsf e_2^{(j)}, \quad\;
\mathsf f^{(i)} R_{ij} = a_{ij}^2 \mathsf f^{(j)} \quad \text{for} \quad 1\le i,j \le 3,
\end{equation}
but for $\mathsf g^{(i)}$,
\begin{equation}\label{g}
\mathsf g^{(i)} R_{ij} = \begin{cases} a_{ij}^2 ( \mathsf g^{(j)} + \mathsf f^{(j)} ) & \text{if} \quad i=2,\;j=3 \quad \text{or} \quad j=2,\;i=3, \\ a_{ij}^2 \mathsf g^{(j)} & \text{for other} \quad 1\le i,j \le 3 . \end{cases}
\end{equation}
Relations~\eqref{eef} are already known from Subsection~\ref{ss:dc}, while \eqref{g} is checked directly.
\end{proof}

\begin{proposition}\label{p:t+}
The block made of double bricks
\begin{equation}\label{double}
\begin{pmatrix} a_{11}\cdot \mathbbm 1_2 & \; a_{12}\cdot \mathbbm 1_2 \; & a_{13}\cdot \mathbbm 1_2 \\
           a_{12}\cdot \mathbbm 1_2 & \; a_{22}\cdot \mathbbm 1_2 \; & a_{23}\cdot T \\
           a_{13}\cdot \mathbbm 1_2 & \; a_{23}\cdot T \; & a_{33}\cdot \mathbbm 1_2 \\ \end{pmatrix},
\end{equation}
where $\mathbbm 1_2$ and~$T$ are as in~\eqref{te}, decomposes into the direct sum of four simple bricks, obtained from~\eqref{cs} by the change $a_{ij}\mapsto a_{ij}^2$ of each entry, and two double bricks~\eqref{T}.
\end{proposition}

\begin{proof}
The same calculations can be applied as in the proof of Proposition~\ref{p:t}, as all the entries in~\eqref{double} belong to a commutative algebra. The r\^ole of~$a_{23}$ is now played by~$a_{23}\cdot T$, but as $T^2=\mathbbm 1_2$, the result is the same as just for two copies of the simple brick~\eqref{cs}.
\end{proof}

\begin{corollary}\label{c:symm}
The evolution---iterative block making---of a symmetric~$A$~\eqref{cs} leads, after $n$ steps, to $2^{2n-1}$ simple and $2^{2n-2}$ double bricks. These are as \eqref{cs} and~\eqref{double}, respectively, but with the change $a_{ij} \mapsto a_{ij}^{2^n}$, \ $1\le i,j\le 3$, in both cases.
\end{corollary}

\begin{proof}
Suppose we have, at some stage, $n_s$ simple bricks and $n_d$ double bricks. If we write these numbers together as a row vector, then Propositions \ref{p:t} and~\ref{p:t+} can be re-written together as the statement that the following transformation happens with this vector after each step:
\begin{equation*}
\begin{pmatrix} n_s & n_d \end{pmatrix} \mapsto \begin{pmatrix} n_s & n_d \end{pmatrix} Q, \qquad \text{where} \quad Q = \begin{pmatrix} 2 & 1 \\ 4 & 2 \end{pmatrix}.
\end{equation*}
One can see then that, indeed,
\begin{equation*}
\begin{pmatrix} 1 & 0 \end{pmatrix} Q^n = \begin{pmatrix} 2^{2n-1} & 2^{2n-2} \end{pmatrix}.
\end{equation*}
\end{proof}

\section{Four dimensions from three dimensions}\label{s:4d}

\subsection{Reducing four dimensions to three, with a special commutative algebra}\label{ss:43}

In four dimensions, we prefer to denote the $4\times 4$ matrix---analogue of ``brick''~$A$~\eqref{A}---as
\begin{equation}\label{B}
B=(b_{ij})_{1\le i,j\le 4}.
\end{equation}
Its entries will belong, in this subsection, to a finite field~$F$ of characteristic~$p$.
As for notation~$A$, it is reserved for a chain of bricks~$B$ situated along the 4th dimension, as shown in Figure~\ref{f:c},
\begin{figure}
 \centering
 \includegraphics[scale=1.1]{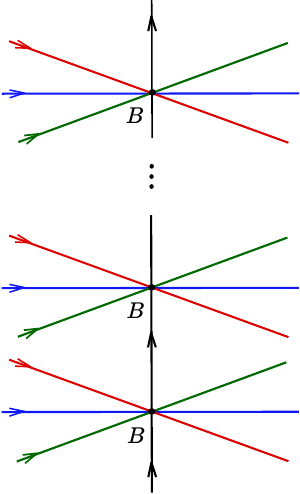}
 \caption{A chain of 4d bricks~$B$ making a 3d matrix~$A$. The 4th axis is vertical and black; three other colors correspond each to its own dimension}
 \label{f:c}
\end{figure}
where the fourth dimension is shown as vertical. Algebraically, we eliminate the spins on the fourth axis, obtaining a three-dimen\-sional~$A$ of the form~\eqref{A} but with entries in a commutative algebra.

The length (height) of the chain in Figure~\ref{f:c} will be denoted~$l$.

We consider two cases of boundary conditions:
\begin{itemize}\itemsep 0pt
 \item[$\mathrm{(a)}$] periodic boundary conditions along the 4th axis: the lowest input is the same as the uppermost output, as in item~\ref{i:p} of Subsection~\ref{ss:pc},
 \item[$\mathrm{(b)}$] each lowest input is zero, with no restrictions on the uppermost outputs, as in item~\ref{i:g} in Subsection~\ref{ss:pc}.
\end{itemize}
In case~$\mathrm{(a)}$, we will need to impose an additional condition on the fourth diagonal entry of~$B$, namely of not being an $l$-th root of unity:
\begin{equation}\label{Bc}
b_{44}^l \ne 1.
\end{equation}

\begin{proposition}\label{p:cg}
In both cases $\mathrm{(a)}$ and~$\mathrm{(b)}$, with condition~\eqref{Bc} in the case~$\mathrm{(a)}$, matrix entries~$a_{ij}$ of~$A$ belong to a commutative algebra~$\mathcal A$. It consists, in case~$\mathrm{(a)}$, of circulant matrices, while in case~$\mathrm{(b)}$---of upper triangular Toeplitz matrices.
\end{proposition}

\begin{proof}
Let $x_i$ denote the row of all the spins at the input arrows going along the $i$-th axis in Figure~\ref{f:c} (counting them from bottom to top), $i=1,2,3,4$, while $y_i$ will denote, similarly, the row of all the spins at the output arrows. We have
\begin{equation}\label{xy}
(y_4)_j = (x_4)_{j+1},\qquad j=1,\ldots,l-1,
\end{equation}
because the $j$-th outgoing arrow going along the 4th axis is the same as the $(j+1)$-th incoming arrow.

In matrix form, \eqref{xy} means
\begin{equation}\label{xym}
\begin{pmatrix} (x_4)_1 & \dots & (x_4)_l \end{pmatrix} = \begin{pmatrix} (y_4)_1 & \dots & (y_4)_l \end{pmatrix} T,
\end{equation}
where
\begin{equation}\label{Tr}
T = \begin{pmatrix} 0 & 1 & & \\ & \ddots & \ddots & \\ & & \ddots & \;1 \\ 1 & & & \;0 \end{pmatrix}\qquad \text{or}\qquad T = \begin{pmatrix} 0 & 1 & & \\ & \ddots & \ddots & \\ & & \ddots & \;1 \\ & & & \;0 \end{pmatrix}
\end{equation}
in the cases $\mathrm{(a)}$ and~$\mathrm{(b)}$, respectively (all entries are zero except those just above the main diagonal, and one more in the lower left corner in the case~$\mathrm{(a)}$).

We now represent $B$ as a block matrix as follows:
\begin{equation}\label{Bb}
 \begin{aligned}
& B=\begin{pmatrix} K & L \\ M & b_{44} \end{pmatrix}, \quad \text{where} \\
K=\begin{pmatrix} b_{11} & b_{12} & b_{13} \\ b_{21} & b_{22} & b_{23} \\ b_{31} & b_{32} & b_{33} \end{pmatrix}, & \qquad L=\begin{pmatrix} b_{14} \\ b_{24} \\ b_{34} \end{pmatrix},\qquad M=\begin{pmatrix} b_{41} & b_{42} & b_{43} \end{pmatrix}.
 \end{aligned}
\end{equation}
We introduce also the following notations:
\begin{align*}
& \mathsf x_i = \begin{pmatrix} (x_i)_1 & \ldots & (x_i)_l \end{pmatrix},\qquad \mathsf y_j = \begin{pmatrix} (y_i)_1 & \ldots & (y_i)_l \end{pmatrix},\qquad i=1,2,3, \\
& X = \begin{pmatrix} \mathsf x_1 & \mathsf x_2 & \mathsf x_3 \end{pmatrix},\qquad Y = \begin{pmatrix} \mathsf y_1 & \mathsf y_2 & \mathsf y_3 \end{pmatrix}, \\
& \tilde K = K\otimes \mathbbm 1_l,\qquad \tilde L = L\otimes \mathbbm 1_l,\qquad \tilde M = M\otimes \mathbbm 1_l.
\end{align*}
That is, $X$ includes all the input spins in the first three dimensions, $Y$ similarly includes output spins, and tilde means here that every entry~$b_{ij}$ in the corresponding matrix is replaced by the $l\times l$ scalar matrix $b_{ij}\cdot \mathbbm 1_l$.

A small linear-algebraic calculation shows then that
\begin{equation}\label{Afrom4}
Y=XA,\qquad \text{where \ } A =  \tilde K+\tilde L (\mathbbm 1_l - b_{44} T)^{-1} T\tilde M .
\end{equation}
Condition~\eqref{Bc} ensures that $(\mathbbm 1_l - b_{44} T)^{-1}$ exists in the case~$\mathrm{(a)}$.

We have obtained thus the desired three-dimen\-sional matrix~$A$, and its explicit form~\eqref{Afrom4} shows that its entries belong to a commutative subalgebra~$\mathcal A$ of the $l\times l$ matrix algebra, namely, $\mathcal A$ is generated by~$T$.
\end{proof}

\subsection{A simple example}\label{ss:mpl}

A simple example occurs when
\begin{equation}\label{b44=0}
b_{44} = 0.
\end{equation}
In case\/~$\mathrm{(a)}$, we get matrix\/~$A$ with entries\/~$a_{ij}$ being, in their turn, the following\/ $l\times l$ matrices:
\begin{equation}\label{a_ij_big}
a_{ij} = \begin{pmatrix} b_{ij} & b_{i4} b_{4j} & & \\ & \ddots & \ddots & \\ & & \ddots & b_{i4} b_{4j} \\ b_{i4} b_{4j} & & & b_{ij} \end{pmatrix}.
\end{equation}
That is, there are entries\/~$b_{ij}$, coming from the brick\/~$B$, all along the main diagonal, entries\/~$b_{i4} b_{4j}$ just above the main diagonal and in the lower left corner, and zeros everywhere else.

In case\/~$\mathrm{(b)}$, the only change is that there is a zero also in the lower left corner: $(a_{ij})_{l1}=0$.

\subsection{Cubic blocks in the simple example}\label{ss:mpl+}

We can now fix some~$l$ and start making blocks like we did in Sections \ref{s:3d} and~\ref{s:3dd}. We see this way that, in our construction, $l$ can be chosen independently of the other three sizes $p^n\times p^n\times p^n$, where $n$ is the number of block making iterations. For the simple example we are presenting, we will content ourself, however, with the case where all four dimensions are the same: $l=p^n$.

Moreover, we put below $p=2$, and assume that condition~\eqref{d},~\eqref{dB} is satisfied.

\begin{proposition}\label{p:tt}
In terms of values~$b_{ij}$, and assuming $l=2^m$, condition~\eqref{d},~\eqref{dB} is written, for matrix~$A$ with entries~\eqref{a_ij_big}, as
\begin{equation}\label{toepb}
( b_{12}+b_{14}b_{42} ) ( b_{23}+b_{24}b_{43} ) ( b_{31}+b_{34}b_{41} ) \ne
( b_{13}+b_{14}b_{43} ) ( b_{32}+b_{34}b_{42} ) ( b_{21}+b_{24}b_{41} ) 
\end{equation}
in the case~$\mathrm{(a)}$, and simply
\begin{equation}\label{trb}
b_{13}b_{21}b_{32} \ne b_{12}b_{23}b_{31}
\end{equation}
in the case~$\mathrm{(b)}$.
\end{proposition}

\begin{proof}
The determinant~$\det M$ of a circulant matrix~$M$ in any finite characteristic~$p$ was calculated in~\cite{Silva}. It can be seen from there that if the size of~$M$ is~$p^n$, then $\det M$ is the $p^n$-th power of the sum of entries in its any row. This, together with~\eqref{a_ij_big}, yields condition~\eqref{toepb}.

As for~\eqref{trb}, it follows of course from the simple fact that the determinant of a triangular Toeplitz matrix is the degree of its diagonal entry.
\end{proof}

\begin{proposition}\label{p:eB}
The (hyper)cube $2^n\times 2^n\times 2^n\times 2^n$ made of bricks \eqref{B},~\eqref{b44=0}, with boundary condition $\mathrm{(a)}$ or~$\mathrm{(b)}$ applied to the fourth dimension, stratifies into $2^n$ independent 3d layers. Each layer is gauge equivalent to the direct sum of\/ $2^{2n-1}+2^{n-1}$ matrices~$\tilde B$ and $2^{2n-1}-2^{n-1}$ matrices~$\tilde B^{\mathrm T}$, where
\begin{equation*}
\tilde B = \begin{cases} \bigl( ( b_{ij}+b_{i4}b_{4j} )^{2^n} \bigr)_{1\le i,j\le 3} & \text{in case } \mathrm{(a)}, \\
                         \bigl( b_{ij}^{2^n} \bigr)_{1\le i,j\le 3} & \text{in case } \mathrm{(b)}. \end{cases}
\end{equation*}
\end{proposition}

\begin{proof}
First, we note that, after $n$ block making iterations, each~$a_{ij}$ turns into a scalar matrix, either $(b_{ij}+b_{i4}b_{4j})^{2^n}\cdot\mathbbm 1_{2^n}$ or $b_{ij}^{2^n}\cdot\mathbbm 1_{2^n}$, in our two respective cases. Then, we combine this with Corollary~\ref{c:d}.
\end{proof}

Note, by the way, that a similar cube made from more general bricks \eqref{B},~\eqref{Bc} stratifies into $2^n$ independent 3d layers as well, because, in both our commutative algebras, any element gives a scalar matrix when raised to the power~$2^n$.

\section{Self-similarity with Boltzmann weights}\label{s:w}

As we said in the beginning of Section~\ref{s:g}, our $\mathcal R$-operators can be considered as either $\mathbb C$-linear or $\mathbb R$-linear. In this section, where we speak of probabilities, we prefer to think of them as $\mathbb R$-linear.

\subsection{Probabilities for permitted spin configurations}\label{ss:(w)p}

For statistical physics, it is of course interesting to look what can happen with our self-similarity if we assign \emph{probabilities} to our permitted spin configurations. For a finite part~$\mathcal L$ of the cubic lattice and such a configuration~$K$ of its spins, the probability must be a real number
\begin{equation}\label{(w)p}
p(K) \in [0, 1] 
\end{equation}
obeying, at least, the following property: if there is a smaller part $\mathcal M\subset \mathcal L$ of the lattice, the probability of a given permitted configuration~$K_{\mathcal M}$ of spins belonging to~$\mathcal M$ is the sum
\begin{equation}\label{(w)M}
p(K_{\mathcal M}) = \sum_{\substack{K_{\mathcal L}\colon \\ K_{\mathcal L}|_{\mathcal M} =K_{\mathcal M}}} p(K_{\mathcal L})
\end{equation}
of probabilities of spin configurations~$K_{\mathcal L}$ on~$\mathcal L$ whose restriction on~$\mathcal M$ is~$K_{\mathcal M}$.

The sum over all permitted configurations must of course be one:
\begin{equation}\label{(w)e}
\sum_{\mathrm{all}\; K_{\mathcal L}} p(K_{\mathcal L}) = 1.
\end{equation}

\subsection{How such probabilities may occur}\label{ss:(w)l}

Matrix entries of our $\mathcal R$-operators are either zeros or unities, we noted it already in Subsection~\ref{ss:n}. The probabilities~$p(K)$ may occur if, for instance,
\begin{enumerate}\itemsep 0pt
 \item\label{i:(w)n} we consider a finite part~$\mathcal L$ of the cubic lattice,
 \item we replace the mentioned unities with non-negative real numbers---``local Boltzmann weights'',
 \item to each permitted spin configuration, we put in correspondence the product of the corresponding local Boltzmann weights,
 \item we calculate the \emph{state sum}---the sum of all mentioned products,
 \item\label{i:(w)k} the probability of a permitted spin configuration is, by definition, its corresponding product divided by the state sum.
\end{enumerate}
In this section we, however, do \emph{not} try to link probabilities~$p(K)$ to local Boltzmann weights, and content ourself with just a few examples of how probabilities~$p(K)$---whatever their origin---can behave with respect to the self-similarity of cubic blocks.

\subsection{The most symmetric probabilities}\label{ss:sym}

Consider a $2\times 2\times 2$ block~$\mathcal L$ of $\mathcal R$-operators, such as in Figure~\ref{f:3d}, with matrix~$A$ entries belonging to just the smallest field~$\mathbb F_2$ of two elements. It looks reasonable to study, at least for the beginning, the \emph{most distinguished}, in some reasonable sense, distributions of probabilities for its permitted spin configurations. Specifically, we propose to study the most \emph{symmetric} probabilities; to be more exact, we require that such a distribution obey the two following kinds of invariance that look natural in this context and that we are going to explain: \emph{restricted translation invariance}, and \emph{``hidden'' invariance} under the group~$\mathrm{GL}(3,\mathbb F_2)$.

\textbf{Restricted translation invariance }means, by definition, the same probabilities for each of the two ``transfer matrices'' corresponding to each of the three directions. A transfer matrix includes, by definition, \emph{one layer} of $\mathcal R$-operators, perpendicular to a given coordinate axis. For instance, a transfer matrix corresponding to the vertical direction is depicted in Figure~\ref{f:tm}.
\begin{figure}
 \centering
 \includegraphics[scale=0.8]{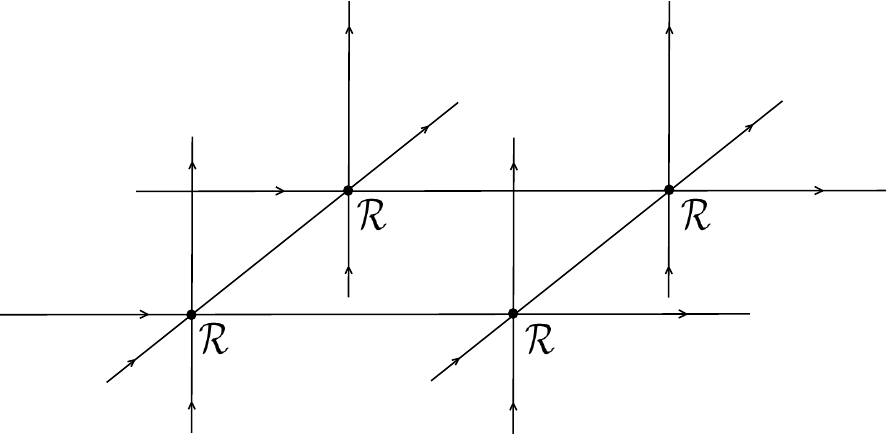}
 \caption{Transfer matrix corresponding to the vertical direction}
 \label{f:tm}
\end{figure}

Here and below, when we speak about ``same probabilities'' for two (or more) lattice fragments, we imply that these fragments can be superposed by a \emph{translation}, so that we can identify naturally their spin configurations.

We will also need one more kind of identification of spin configurations, appearing the following way. As $x^2=x$ for $x\in \mathbb F_2$, the 2nd, 3rd and 4th direct summands in~\eqref{RFrob} are the same in our case as the initial matrix~$A$~\eqref{A}. Hence, we can consider any of the mentioned summands separately as a new vertex, and identify, if needed, its spin configurations with those corresponding to any one vertex of the initial lattice.

Now we recall where the mentioned direct summands acted in Subsection~\ref{ss:dc}: the 2nd summand acted in the linear span of vectors $\mathsf e_1^{(1)}$, $\mathsf e_1^{(2)}$ and~$\mathsf e_1^{(3)}$; the 3rd summand---in the linear span of vectors $\mathsf e_2^{(1)}$, $\mathsf e_2^{(2)}$ and~$\mathsf e_2^{(3)}$; and the 4th, similarly,---in the linear span of vectors $\mathsf e_3^{(1)}$, $\mathsf e_3^{(2)}$ and~$\mathsf e_3^{(3)}$. Let $M\in \mathrm{GL}(3,\mathbb F_2)$---that is, let $M$ be a non-degenerate $3\times 3$ matrix with entries in~$\mathbb F_2$, and define new vectors $\tilde {\mathsf e}_i^{(j)}$, \ $i,j=1,2,3$, as follows:
\begin{equation}\label{wM}
\begin{pmatrix} \tilde {\mathsf e}_1^{(1)} & \tilde {\mathsf e}_1^{(2)} & \tilde {\mathsf e}_1^{(3)} \\ 
                \tilde {\mathsf e}_2^{(1)} & \tilde {\mathsf e}_2^{(2)} & \tilde {\mathsf e}_2^{(3)} \\ 
                \tilde {\mathsf e}_3^{(1)} & \tilde {\mathsf e}_3^{(2)} & \tilde {\mathsf e}_3^{(3)} \end{pmatrix}
= M \begin{pmatrix} \mathsf e_1^{(1)} & \mathsf e_1^{(2)} & \mathsf e_1^{(3)} \\ 
                    \mathsf e_2^{(1)} & \mathsf e_2^{(2)} & \mathsf e_2^{(3)} \\ 
                    \mathsf e_3^{(1)} & \mathsf e_3^{(2)} & \mathsf e_3^{(3)} \end{pmatrix}.
\end{equation}

\begin{proposition}\label{p:(w)e}
The block~$R$~\eqref{R} acts the same way in the space spanned by vectors $\tilde {\mathsf e}_i^{(1)}$, $\tilde {\mathsf e}_i^{(2)}$ and~$\tilde {\mathsf e}_i^{(3)}$, for any $i=1,2,3$, as it did for vectors without tildes, namely, like the initial~$A$ in our case of field\/~$\mathbb F_2$.
\end{proposition}

\begin{proof}
This follows immediately from the fact that (the restriction of) $R$ acts on the rows from the \emph{right}, and its action is the same for all the rows of the rightmost matrix in~\eqref{wM}, while matrix~$M$ acts from the \emph{left}.
\end{proof}

\textbf{``Hidden'' invariance }of probabilities holds, by definition, provided they do not change under the transformation $\bigl(\mathsf e_i^{(j)}\bigr) \to \bigl(\tilde{\mathsf e}_i^{(j)}\bigr)$ of spin configurations for block~$R$~\eqref{R} given by~\eqref{wM}, for any matrix $M\in \mathrm{GL}(3,\mathbb F_2)$.

There are
\begin{equation*}
|\mathbb F_2|^{\mathrm{number\; of\;input\;edges}}=2^{12}=4096
\end{equation*}
permitted spin configurations for Figure~\ref{f:3d}. Hence, if we put a real number in correspondence to each permitted spin configuration, such 4096-tuples form the linear space~$\mathbb R^{4096}$. We will take the liberty of referring to the elements of this~$\mathbb R^{4096}$ as `probabilities', in quotation marks. These `probabilities' differ from probabilities proper in that we forget, for the former, about requirements \eqref{(w)p} and~\eqref{(w)e}. Formula~\eqref{(w)M} remains valid, however.

Our two symmetries are, in this terminology, $\mathbb R$-linear restrictions on `probabilities' that single out a linear subspace $W\subset \mathbb R^{4096}$ regarded as consisting of the most symmetric `probabilities'.

\subsection[Factor spaces of~$W$]{Factor spaces of~$\bm{W}$}\label{ss:fW}

Suppose we have declared some permitted spin configurations for Figure~\ref{f:3d} \emph{equivalent}, and that the `probability' of such equivalence class~$\mathcal C$ is the sum
\begin{equation}\label{(w)C}
p(\mathcal C) = \sum _{K\in \mathcal C} p(K).
\end{equation}
We have already encountered one example of such situation in formula~\eqref{(w)M}.

All `probabilities'~\eqref{(w)C} arising this way form, as one can see, a \emph{factor} space of the initial space of `probabilities' given for all permitted spin configurations. We will take~$W$ for this initial space, and we will be interested in its following factor spaces:
\begin{itemize}\itemsep 0pt
 \item $W_0$, corresponding to $\mathcal M$ in~\eqref{(w)M} being just one vertex. It does not matter which exactly vertex we take, due to the restricted translation invariance,
 \item $W_1$, where spin configurations are declared equivalent provided they give the same input $\mathbb F_2$-vector for the \emph{first} direct summand in~\eqref{RFrob},
 \item $W_2$, $W_3$, $W_4$ are similar factor spaces for the 2nd, 3rd and 4th summands in~\eqref{RFrob}, respectively. These three spaces are, however, the same, due to the ``hidden'' invariance of~$W$, so we speak below just of~$W_2$.
\end{itemize}

Recall that the 2nd, 3rd and 4th summands in~\eqref{RFrob} are here the same as the initial matrix~$A$~\eqref{A}, while the first summand is~$A^{\mathrm T}$.

\subsection{Examples: dimensions of spaces}\label{ss:(w)d}

\begin{example}\label{e:a,at}
The following two matrices are each the transpose of the other:
\begin{equation}\label{a,at}
A = \begin{pmatrix} 1 & 1 & 0 \\ 0 & 0 & 1 \\ 1 & 0 & 0 \end{pmatrix} \qquad \text{or} \qquad
A = \begin{pmatrix} 1 & 0 & 1 \\ 1 & 0 & 0 \\ 0 & 1 & 0 \end{pmatrix}.
\end{equation}
The dimensions of spaces $W$, $W_0$, $W_1$ and~$W_2$ for each of these matrices turn out to be as follows:
\begin{equation*}
\dim W=13, \qquad \dim W_0=2, \qquad \dim W_1=\dim W_2=5.
\end{equation*}
\end{example}

\begin{example}\label{e:m3,m3t}
Next, we take
\begin{equation}\label{m3,m3t}
A = \begin{pmatrix} 1 & 1 & 1 \\ 0 & 1 & 1 \\ 1 & 1 & 0 \end{pmatrix} \qquad \text{or} \qquad
A = \begin{pmatrix} 1 & 0 & 1 \\ 1 & 1 & 1 \\ 1 & 1 & 0 \end{pmatrix},
\end{equation}
again being the transposes of one another. The dimensions are now
\begin{equation*}
\dim W=9, \qquad \dim W_0=2, \qquad \dim W_1=\dim W_2=3.
\end{equation*}
\end{example}

\begin{example}\label{e:tt,t}
The dimensions of~$W$ need \emph{not} coincide, however, for a matrix and its transpose. For
\begin{equation}\label{(w)tt}
A = \begin{pmatrix} 1 & 1 & 1 \\ 0 & 0 & 1 \\ 1 & 0 & 0 \end{pmatrix},
\end{equation}
the dimensions are
\begin{equation*}
\dim W=13, \qquad \dim W_0=2, \qquad \dim W_1=\dim W_2=5,
\end{equation*}
while if we take for~$A$ the transpose of~\eqref{(w)tt}, that is,
\begin{equation}\label{(w)t}
A = \begin{pmatrix} 1 & 0 & 1 \\ 1 & 0 & 0 \\ 1 & 1 & 0 \end{pmatrix},
\end{equation}
then the dimensions are
\begin{equation*}
\dim W=17, \qquad \dim W_0=2, \qquad \dim W_1=\dim W_2=5.
\end{equation*}
\end{example}

\begin{example}\label{e:dgn}
Finally, for any of the two \emph{degenerate} matrices
\begin{equation}\label{dgn}
A = \begin{pmatrix} 1 & 1 & 0 \\ 0 & 1 & 1 \\ 1 & 0 & 1 \end{pmatrix} \qquad \text{or} \qquad
A = \begin{pmatrix} 1 & 0 & 1 \\ 1 & 1 & 0 \\ 0 & 1 & 1 \end{pmatrix},
\end{equation}
the dimensions are
\begin{equation*}
\dim W=9, \qquad \dim W_0=2, \qquad \dim W_1=\dim W_2=3.
\end{equation*}
\end{example}

\subsection{The self-similarity}\label{ss:(w)y}

Each of the spaces $W_0$, $W_1$ and~$W_2$ corresponds to one vertex, with three input edges. Hence, it is a subspace of the space~$\mathbb R^8$ of all possible `probabilities' for one edge.

In each case of Subsection~\ref{ss:(w)d}, the space~$W_0$ turned out to be the same, namely, the linear span of the following vectors from~$\mathbb R^8$:
\begin{equation}\label{(w)h}
\begin{pmatrix} 1 & 0 & 0 & 0 & 0 & 0 & 0 & 0 \end{pmatrix} \qquad \text{and} \qquad \begin{pmatrix} 0 & 1 & 1 & 1 & 1 & 1 & 1 & 1 \end{pmatrix},
\end{equation}
where the first component corresponds to the input basis vector $\begin{pmatrix} 0 & 0 & 0 \end{pmatrix} \in \mathbb F_2^3$, while the rest of the components belong to all other input $\mathbb F_2$-vectors. The order of these other components clearly does not matter in~\eqref{(w)h}; actually, we use the lexicographic order in our calculations: $\begin{pmatrix} 0 & 0 & 0 \end{pmatrix} \in \mathbb F_2^3$ goes first, then $\begin{pmatrix} 0 & 0 & 1 \end{pmatrix} \in \mathbb F_2^3$, and so on.

It turns out also that both vectors~\eqref{(w)h} enter, in all our examples, in both $W_1$ and~$W_2$ as well. Of these, $W_2$ looks most interesting for self-similarity, because its corresponding direct summand in~\eqref{RFrob} is the same as the initial~$A$. We can write thus $W_0\subset W_2$, identifying the corresponding spin configurations as explained in Subsection~\ref{ss:sym}.

A vector $w\in W$ induces, according to~\eqref{(w)C}, some vectors $w_0\in W_0$ and~$w_2\in W_2$. These $w_0$ and~$w_2$, taken for all~$w$, are not independent: a remarkable fact is that there are \emph{two linear dependences} between them, for any example of Subsection~\ref{ss:(w)d}. Moreover, these linear dependences turn out to be such that
\begin{equation}\label{w0w2}
\text{if\, additionally}\quad w_2\in W_0, \quad \text{then}\quad w_2 = w_0.
\end{equation}

The requirement $w_2\in W_0$ is actually quite natural, because when we aggregate the eight vertices in Figure~\ref{f:3d} into one ``block vertex'', we probably want to take then a block of eight such ``block vertices'' and proceed with it the same way, and we probably expect our ``block probabilities'' to have again the symmetries of Subsection~\ref{ss:sym}. For this block of blocks, ``old'' $W_2$ acquires, however, the r\^ole of new~$W_0$.

In this sense, the self-similarity, shown experimentally for all examples of Subsection~\ref{ss:(w)d}, means that, for any $\alpha \in [0,1]$ and $\beta = (1-\alpha)/7$, if the probability (quotation marks can be omitted now) distribution for a vertex is given by $\begin{pmatrix} \alpha & \beta & \beta & \beta & \beta & \beta & \beta & \beta \end{pmatrix}$, then it remains \emph{the same for block vertices}.

\section{Discussion}\label{s:d}

We have considered an algebraic analogue of Kadanoff--Wilson theory, related to fields of finite characteristics and showing a surprising generality and a surprisingly simple kind of self-similarity when making spin blocks.

This way, we can calculate some statistical quantities related to cubic lattices in different dimensions. These quantities are---at the moment---integer-valued, namely the numbers of ``permitted spin configurations'', for given relations between spins around each vertex and some types of boundary conditions.

Further work in this direction may consist in calculating and exploring statistical quantities related to three-dimen\-sional cubic lattice but more complicated commutative algebras, perhaps even infinite-dimen\-sional.

There are indications that the theory will become even more interesting with adding a \emph{non-homogeneous} term to the dependence between the input and output spins of a vertex, like it was done in~\cite{Hietarinta}:
\begin{equation*}
\mathsf x_{\mathrm{out}} = \mathsf x_{\mathrm{in}} A + \mathsf b \quad\; \text{instead\, of\, just}\quad\; \mathsf x_{\mathrm{out}} = \mathsf x_{\mathrm{in}} A,
\end{equation*}
where $\mathsf x_{\mathrm{in}}$ and~$\mathsf x_{\mathrm{out}}$ are vectors of input and output spins taking values in a finite field, $A$ a matrix, and $\mathsf b$ a vector.

Finally, Section~\ref{s:w} gives the positive answer to the question whether our self-similarity has anything to do with non-trivial---real positive---Boltzmann weights. It would be interesting to investigate the properties of our `most symmetric probability distribution' for the infinite cubic lattice, understanding that the symmetries of Subsection~\ref{ss:sym} apply at every block making step. More generally, the algebraic structure of self-similarity with Boltzmann weights needs theoretical investigation.

\end{document}